\documentclass[10pt]{amsart}
\usepackage{amsmath,amssymb}

\newtheorem{theorem}{Theorem}[section]

\newtheorem{corollary}[theorem]{Corollary}
\newtheorem{lemma}[theorem]{Lemma}

\newtheorem{proposition}[theorem]{Proposition}

\numberwithin{equation}{section}

\newcommand{\R}{\mathbb{R}}

\newcommand{\C}{\mathbb{C}}
\newcommand{\T}{\mathbb{T}}
\newcommand{\dis}{\displaystyle}

\begin{document}

\title[Weighted energy problem on the unit circle]
{Weighted energy problem on the unit circle}%
\author{Igor E. Pritsker}%

\address{Department of Mathematics, Oklahoma State University, Stillwater, OK
74078, U.S.A.}%
\email{igor@math.okstate.edu}




\begin{abstract}

We solve the weighted energy problem on the unit circle, by
finding the extremal measure and describing its support.
Applications to polynomial and exponential weights are also
included.

\end{abstract}

\maketitle


\section{Introduction and main results}

Let $w\not\equiv 0$ be a continuous nonnegative function on the
unit circle $\T:=\{z:|z|=1\}$, and set
\begin{align} \label{1.1}
Q(z):= - \log w(z).
\end{align}
Let ${\mathcal M}(\T)$ be the space of positive unit Borel
measures supported on $\T$. For any measure $\mu\in {\mathcal
M}(\T)$, we define the energy functional
\begin{align} \label{1.2}
I_{w} (\mu)&:= \int\!\!\int \log \dis\frac{1}{|z-t|w(z)w(t)} \ d
\mu(z)d \mu(t) \\ \nonumber &= \int\!\!\int \log \dis
\frac{1}{|z-t|} \ d \mu(z)d \mu(t) + 2 \int Q(t)\, d\mu(t),
\end{align}
and consider the minimum energy problem
\begin{equation} \label{1.3}
V_{w}:= \dis\inf_{\mu \in {\mathcal M}(\T)} I_{w}(\mu).
\end{equation}
For a general reference on potential theory with external fields,
or weighted potential theory, one should consult the book of Saff
and Totik \cite{ST}. It follows from Theorem I.1.3 of \cite{ST}
that $V_{w}$ is finite, and there exists a unique equilibrium
measure $\mu_{w} \in {\mathcal{M}}(\T)$ such that $I_{w} (\mu_{w})
= V_{w}$. Thus $\mu_w$ minimizes the energy functional \eqref{1.2}
in presence of the external field $Q$ generated by the weight $w$.
Furthermore, we have for the potential of $\mu_w$ that
\begin{equation} \label{1.4}
U^{\mu_{w}}(z)+Q(z) \geq F_{w}, \qquad z \in \T,
\end{equation}
and
\begin{equation} \label{1.5}
U^{\mu_{w}}(z)+Q(z) = F_{w}, \qquad z \in S_w,
\end{equation}
where $U^{\mu_w}(z):=-\dis\int\log|z-t|\,d\mu_w(t), \ F_{w}:=V_{w}
- \dis\int Q(t) d\mu_{w}(t)$ and $S_w:={\rm supp}\, \mu_{w}$ (see
Theorems I.1.3 and I.5.1 in \cite{ST}) . The weighted capacity of
$\T$ is defined by
\begin{equation} \label{1.6}
\textup{cap}(\T,w):=e^{-V_w}.
\end{equation}
If $w\equiv 1$ on $\T$, then we obtain the classical logarithmic
capacity $\textup{cap}(\T)=1,$ and the equilibrium measure
$dt/(2\pi),\ e^{it}\in\T.$

The energy problems with external fields on subsets of the real
line were treated in many papers, see \cite{ST} for a survey and
references. The purpose of this paper is twofold: we provide a
general solution to the weighted energy problem on the unit
circle, and we also simplify the previously known arguments used
in the real line case. Our method applies on the real line too,
which leads to shorter proofs and generalizations of the results
in \cite{DKM} and \cite{ST}.

We give below an explicit form of the equilibrium measure and
describe its support for the weighted energy problem on $\T$.
Throughout the paper, we use the notation $Q(t):=Q(e^{it}).$

\begin{theorem} \label{thm1.1}
Suppose that $Q\in C^{1+\varepsilon}(U),$ where $U$ is an open
neighborhood of $S_w$ in $\T.$ Then
$d\mu_w(e^{i\theta})=f(\theta)\, d\theta,$ where $f\in
L_{\infty}([0,2\pi)).$ Furthermore, the density $f(\theta)$
satisfies the equation
\begin{equation} \label{1.7}
f^2(\theta) = \left(\frac{Q'(\theta)}{\pi}\right)^2 -
\frac{1}{\pi^2}\int_0^{2\pi} Q'(t)f(t)\cot\frac{\theta-t}{2}\,dt +
\frac{1}{4\pi^2}
\end{equation}
for a.e. $e^{i\theta}\in S_w,$ where the integral in \eqref{1.7}
is understood in the principal value sense.
\end{theorem}

\begin{corollary} \label{cor1.2}
Theorem \ref{thm1.1} also holds with \eqref{1.7} replaced by
\begin{equation} \label{1.8}
f^2(\theta) = \frac{1}{\pi^2}\int_0^{2\pi} (Q'(\theta)- Q'(t))f(t)
\cot\frac{\theta-t}{2}\,dt - \left(\frac{Q'(\theta)}{\pi}\right)^2
+ \frac{1}{4\pi^2}
\end{equation}
for a.e. $e^{i\theta}\in S_w.$

Let $p(\theta)$ be the right hand side of \eqref{1.8}. If $Q\in
C^2(U)$ then $p\in C(U)$ and $f:=\sqrt{p}\in C(S_w).$ Furthermore,
$S_w$ is the closure in $\T$ of the open set $\{e^{i\theta}\in\T:
p(\theta)>0\}$. Hence $f(\theta)$ vanishes at the endpoints of
$S_w.$

Moreover, if $Q$ is real analytic on $U$, then $S_w$ is a finite
union of closed arcs of $\T$.
\end{corollary}

\begin{corollary} \label{cor1.2a}
If $S_w=\T$ under the assumptions of Theorem \ref{thm1.1}, then
\[
f(\theta) =  \frac{1}{2\pi} - \frac{1}{2\pi^2}\int_0^{2\pi} Q'(t)
\cot\frac{\theta-t}{2}\,dt, \qquad e^{i\theta}\in \T.
\]
In particular, $f$ is H\"{o}lder continuous on $\T.$
\end{corollary}

We mention another instance when the structure of $S_w$ is clear,
which parallels the real line case (see Theorem IV.1.10(b) of
\cite{ST}).

\begin{proposition} \label{prop1.3}
If $Q(t)$ is convex on $(\alpha,\beta)\subset\R,\ \beta-\alpha\le
2\pi$, then the intersection of $S_w$ with the arc
$(e^{i\alpha},e^{i\beta})\subset\T$ is either an arc or empty set.
\end{proposition}

The support $S_w$ plays crucial role in determining the
equilibrium measure $\mu_w$ itself, as well as other components of
this weighted energy problem. Indeed, if $S_w$ is known then
$\mu_w$ can be found as a solution of the singular integral
equation
\[
\int\log\frac{1}{|z-\zeta|}\,d\mu(\zeta)-\log w(z) = F, \qquad z
\in S_w,
\]
where $F$ is a constant (cf. \eqref{1.5} and \cite[Ch. IV]{ST}).
Applying the results of \cite{Mus} (see also \cite{Mu41}), we
solve this integral equation and obtain the following theorem.

\begin{theorem} \label{thm1.4}
Let $Q\in C^{2}(U),$ where $U$ is an open neighborhood of $S_w$ in
$\T.$ Assume that $S_w$ consists of $K$ arcs $\Gamma_k\subset \T$,
$K\ge 1$, with the endpoints $a_k=e^{i\alpha_k}$ and
$b_k=e^{i\beta_k}$ such that $
\alpha_1<\beta_1<\alpha_2<\beta_2<\ldots<\alpha_K<\beta_K,\ \
\beta_K-\alpha_1 < 2\pi$. Set $R(z):=\prod_{k=1}^K
(z-a_k)(z-b_k),$ and consider the branch of $\sqrt{R(z)}$ defined
in the domain ${\overline\C}\setminus S_w$ by $\lim_{z\to\infty}
\sqrt{R(z)}/z^K =1.$ By the values of $\sqrt{R(z)}$ on $S_w$, we
understand the limiting values from inside the unit disk.  Let
\begin{equation} \label{1.9}
F(z):=\frac{\sqrt{R(z)}}{\pi i} \int_{S_w}
\frac{g(\zeta)\,d\zeta}{\sqrt{R(\zeta)}(\zeta-z)},
\end{equation}
where
\[
g(e^{it}):=\frac{i}{\pi} Q'(t) + \frac{1}{2\pi},\qquad e^{it}\in
S_w,
\]
and the integral in \eqref{1.9} is the Cauchy principal value.
Then the density $f$ of $\mu_w$ is given by
\begin{equation} \label{1.10}
f(t)=F(e^{it}),\qquad e^{it}\in S_w.
\end{equation}
Furthermore, the following equations are satisfied
\begin{equation} \label{1.11}
\int_{S_w} \frac{z^k g(z)\,dz}{\sqrt{R(z)}} = 0, \qquad
k=0,\ldots,K-1,
\end{equation}
\begin{equation} \label{1.12}
\int_{S_w} f(t)\,dt = 1,
\end{equation}
and
\begin{align} \label{1.13}
\int_{\beta_k}^{\alpha_{k+1}} F(e^{it})\, dt =
\frac{\alpha_{k+1}-\beta_k}{2\pi} +
\frac{Q(\alpha_{k+1})-Q(\beta_k)}{\pi}\,i, \qquad k=1,\ldots,K,
\end{align}
where we assume that $\alpha_{K+1}=\alpha_1+2\pi.$
\end{theorem}

Note that equations \eqref{1.11}-\eqref{1.13} may be used to find
the endpoints of $S_w.$

\section{Applications}

Consider the weight function
\begin{align} \label{2.1}
w(z):=\prod_{j=1}^J |z-z_j|^{\lambda_j},
\end{align}
where $\lambda_j>0$, $z_j\in\C$ and $z_j\neq 0,\ j=1,\ldots,J$. In
some special cases, such weights on disks were previously treated
in \cite{PV98}, with applications to the weighted polynomial
approximation. In general, one can express the equilibrium measure
$\mu_w$ for $w$ of \eqref{2.1} as a linear combination of harmonic
measures, which was done in \cite{Pr00} and \cite{Pr02}. The
complete solution of the weighted energy problem on the unit
circle given below was first found in \cite{PV04}, in connection
with a number theoretic problem on heights of some subspaces of
polynomials (see \cite{BV} for background).

\begin{theorem} \label{thm2.1}
For the weight $w$ of \eqref{2.1}, the support $S_w$ consists of
$K \le J$ arcs. If $S_w=\T$ then we have that
\begin{align} \label{2.2}
d\mu_w(e^{it}) = \frac{1}{2\pi}\left(1+\sum_{j=1}^J \lambda_j -
\sum_{j=1}^J \lambda_j \frac{\left| |z_j|^2-1 \right|}
{|e^{it}-z_j|^2} \right)dt.
\end{align}
If $S_w\neq\T$ then we set
\[
F_1(z):=\frac{\sqrt{R(z)}}{2\pi} \sum_{j=1}^J \lambda_j
\left(\frac{z_j}{(z_j-z)\sqrt{R(z_j)}} + \frac{\bar
z_j^{-1}}{(\bar z_j^{-1}-z)\sqrt{R(\bar z_j^{-1})}}\right),
\]
where we use the notation of Theorem \ref{thm1.4}. The equilibrium
measure $\mu_w$ is given in this case by
\begin{equation} \label{2.3}
d\mu_{w}(e^{it})=F_1(e^{it})\, dt, \qquad e^{it} \in S_w,
\end{equation}
where the values of $\sqrt{R(z)}$ on $S_w$ are the limiting values
from inside the unit disk.

Furthermore, the endpoints of $S_w$ satisfy the equations
\begin{eqnarray} \label{2.4}
 & \dis\sum_{j=1}^J \lambda_j  \left(\frac{z_j^k}{\sqrt{R(z_j)}} +
\frac{\bar z_j^{-k}}{\sqrt{R(\bar z_j^{-1})}}\right)=0, \qquad
 k=1,\ldots,K-1,\nonumber\\ \\ \nonumber &\dis \sum_{j=1}^J \lambda_j
 \left(\frac{z_j^K}{\sqrt{R(z_j)}} + \frac{\bar z_j^{-K}}{\sqrt{R(\bar
z_j^{-1})}}\right) = 1 + \sum_{j=1}^J \lambda_j,
\end{eqnarray}
and the equations
\begin{equation} \label{2.5}
\int_{\beta_k}^{\alpha_{k+1}} F_1(e^{it})\, dt=0, \qquad
k=1,\ldots,K.
\end{equation}
\end{theorem}

\bigskip

Our second application is related to the exponential weights of
the form
\begin{equation} \label{2.6}
w(e^{i\theta}) = e^{-t_M(\theta)}, \ \mbox{where}\quad
t_M(\theta):=\sum_{m=-M}^M c_m e^{im\theta}
\end{equation}
is a real valued trigonometric polynomial of degree $M$. A typical
example of such weight is given by $w(z) = |e^{-cz}|,\ c\in\R,$
which was studied in \cite{PV97} and \cite{PV98} on disks and on
the Szeg\H{o} domain. The same weight appeared in a problem on the
longest increasing subsequence of random permutations, see
\cite{BDJ}. The general solution of the weighted energy problem on
the unit circle is given below.

\begin{theorem} \label{thm2.2}
For the weight $w$ of \eqref{2.6}, the support $S_w$ consists of
$K \le M$ arcs. If $S_w=\T$ then we have that
\begin{align} \label{2.7}
d\mu_w(e^{it}) = \left(\frac{1}{2\pi} - \frac{1}{\pi}
\sum_{m=-M}^M m\, {\rm sgn}(m)\, c_m e^{im\theta} \right)dt.
\end{align}
If $S_w\neq\T$ then we set
\[
F_2(z):=\frac{\sqrt{R(z)}}{\pi} \left( \sum_{m=K}^M m c_m
s_{m+1}(z) - \sum_{m=-M}^{-1} m c_m r_{-m-1}(z) \right),
\]
in the notation of Theorem \ref{thm1.4}, where
\[
\frac{1}{\sqrt{R(\zeta)}(\zeta-z)} = \sum_{k=K+1}^{\infty}
\frac{s_k(z)}{\zeta^k} \quad \mbox{and} \quad
\frac{1}{\sqrt{R(\zeta)}(\zeta-z)} = \sum_{k=0}^{\infty} r_k(z)
\zeta^k,
\]
respectively near $\infty$ and near $0$. The measure $\mu_w$ is
given by
\begin{equation} \label{2.8}
d\mu_{w}(e^{it})=F_2(e^{it})\, dt, \qquad e^{it} \in S_w,
\end{equation}
where the values of $\sqrt{R(z)}$ on $S_w$ are the limiting values
from inside the unit disk.

Furthermore, the endpoints of $S_w$ satisfy \eqref{1.11} with
\[
g(z)=\frac{1}{2\pi} - \frac{1}{\pi} \sum_{m=-M}^M m c_m z^m,
\]
and \eqref{1.13} with $F=F_2+g.$
\end{theorem}

Note that
\[
r_k(z) = \frac{1}{k!}\, \frac{d^k}{d\zeta^k} \left. \left(
\frac{1}{\sqrt{R(\zeta)}(\zeta-z)} \right) \right|_{\zeta=0},\quad
k\ge 0,
\]
and
\[
s_k(z) = \frac{1}{k!}\, \frac{d^k}{d\zeta^k} \left. \left(
\frac{\zeta^{K+1}}{\sqrt{\zeta^{2K}R(1/\zeta)}(1-z\zeta)} \right)
\right|_{\zeta=0},\quad k\ge 0.
\]
Hence $s_k(z)$ is a polynomial in $z$ of degree at most $k$, and
$r_k(z)$ is a polynomial in $1/z$ of degree at most $k+1$. Also,
it is clearly possible to evaluate the integrals in \eqref{1.11}
(for this function $g$) by using the residues at $0$ and $\infty.$

\section{Proofs}

The proof of Theorem \ref{thm1.1} requires the following lemma.

\begin{lemma} \label{lem3.1}
The weighted equilibrium measure $\mu_w$ is absolutely continuous
with respect to the arclength on $\T$, and
\[
d\mu_w(e^{it})=f(t)\,dt, \quad e^{it} \in \T,
\]
where $f\in L_{\infty}\left([0,2\pi)\right).$
\end{lemma}

\begin{proof}
We shall show that $U^{\mu_w}$ is Lipschitz continuous in $\C$,
which implies that the directional derivatives of $U^{\mu_w}$
exist a.e. on $\T$, and
\[
d\mu_w(e^{it})=-\frac{1}{2\pi} \left(\frac{\partial
U^{\mu_w}}{\partial {\bf n}_-}(e^{it}) + \frac{\partial
U^{\mu_w}}{\partial {\bf n}_+}(e^{it})\right)\,dt =: f(t)\,dt
\]
for a.e. $e^{it} \in \T,$ by Theorem II.1.5 of \cite{ST}, where
${\bf n}_-$ and ${\bf n}_+$ are the outer and the inner normals to
$\T$. Clearly, the normal derivatives of $U^{\mu_w}$ are bounded
by the Lipschitz constant, so that we obtain $f\in
L_{\infty}\left([0,2\pi)\right).$

Recall that $Q(t)=-\log w(e^{it})$ is a $C^{1+\varepsilon}$
function in an open neighborhood $U$ of $S_w$ in $\T.$ We can
modify $w(e^{it})$ so that for the resulting function $v(e^{it})$
we still have
\[
U^{\mu_{w}}(z)-\log v(z) = F_{w}, \qquad z \in S_w,
\]
and
\[
U^{\mu_{w}}(z)-\log v(z) \ge F_{w}, \qquad z \in \T,
\]
and we also have $\log v(e^{it})\in C^{1+\varepsilon}(\T).$
Theorem I.3.3 of \cite{ST} then implies that $\mu_v=\mu_w$ and
$F_v=F_w.$ Thus we can work with $v$ instead. Indeed, such
modification is possible by \eqref{1.4} and \eqref{1.5}, because
$S_w$ is a compact set contained in $U$. Hence we can find an open
cover $O\subset U$ for $S_w$, consisting of finitely many open
arcs. Then we set $v|_O=w|_O,$ and modify $w$ to $v$ on
$\T\setminus O$ (consisting of finitely many closed arcs) in such
a way that (cf. \eqref{1.4})
\[
U^{\mu_{w}}(z)-\log v(z) \ge F_{w}, \qquad z \in \T\setminus O,
\]
and $\log v(e^{it})\in C^{1+\varepsilon}(\T).$

Let $u$ be a solution of the Dirichlet problem in the unit disk
$D$ for the boundary data $\log v(e^{it})+F_w$. Then $u\in
C^{1+\varepsilon}(\overline{D})$ by Privalov's theorem (see $\S$5
of Chap. IX in \cite{Go}). Since $u|_{\T}\le U^{\mu_w}|_{\T}$ by
our construction, we obtain that
\[
u(z)\le U^{\mu_w}(z), \qquad z\in \overline{D},
\]
as $U^{\mu_w}$ is superharmonic. In the proof of Lipschitz
continuity of $U^{\mu_w}$, we first consider $z\in S_w$ and
$\zeta\in \C$, and follow an idea of G\"otz \cite{Got}. Since
$u|_{S_w}=U^{\mu_w}|_{S_w}$, it is immediate that
\[
U^{\mu_w}(z) - U^{\mu_w}(\zeta) \le u(z) - u(\zeta) \le
C|z-\zeta|, \quad z\in S_w,\ \zeta\in\overline D,
\]
where $C$ is the Lipschitz constant for $u$ on $\overline D$. Note
that $U^{\mu_w}(1/\bar\zeta)=U^{\mu_w}(\zeta)+\log|\zeta|,\
\zeta\neq 0,$ because $S_w\subset\T.$ Hence we have from the above
estimate that
\begin{equation} \label{3.1}
U^{\mu_w}(z) - U^{\mu_w}(\zeta) \le (C+1)|z-\zeta|, \quad z\in
S_w,\ \zeta\in\C.
\end{equation}
In order to prove a matching estimate from below, we consider a
nearest point $\zeta^*\in S_w$ for $\zeta,$ i.e.,
dist$(\zeta,S_w)=|\zeta-\zeta^*|=:r.$ Then
\begin{align} \label{3.2}
U^{\mu_w}(z) - U^{\mu_w}(\zeta) &= u(z) - u(\zeta^*) +
U^{\mu_w}(\zeta^*) - U^{\mu_w}(\zeta) \\ \nonumber &\ge
-C|z-\zeta^*| + U^{\mu_w}(\zeta^*) - U^{\mu_w}(\zeta), \quad z\in
S_w,\ \zeta\in\C.
\end{align}
Using the area mean-value inequality, we obtain that
\begin{align} \label{3.3}
U^{\mu_w}(\zeta^*) &\ge \frac{1}{\pi(2r)^2}\, \int_{D_{2r}(\zeta^*)}
U^{\mu_w}(x+iy)\,dxdy \\
\nonumber &= \frac{1}{4\pi r^2}\, \int_{D_{2r}(\zeta^*)\setminus
D_r(\zeta)} U^{\mu_w}(x+iy)\,dxdy + \frac{1}{4}\,
U^{\mu_w}(\zeta),
\end{align}
where the second term comes from the mean-value property for the
harmonic function $U^{\mu_w}$ in $D_r(\zeta)$. Note that
$U^{\mu_w}(\xi) \ge U^{\mu_w}(\zeta^*) -(C+1)|\zeta^*-\xi|, \
\xi\in\C,$ by \eqref{3.1}. Hence \eqref{3.3} implies that
\[
U^{\mu_w}(\zeta^*) \ge \frac{4\pi r^2 -\pi r^2}{4\pi r^2}
\left(U^{\mu_w}(\zeta^*)-(C+1)2r\right) + \frac{1}{4}\,
U^{\mu_w}(\zeta),
\]
and that
\[
U^{\mu_w}(\zeta^*) - U^{\mu_w}(\zeta) \ge -6(C+1)r.
\]
Applying this in \eqref{3.2}, we have
\begin{align*}
U^{\mu_w}(z) - U^{\mu_w}(\zeta) &\ge -C|z-\zeta^*| -
6(C+1)|\zeta-\zeta^*| \\ &\ge -C|z-\zeta| - 7(C+1)|\zeta-\zeta^*|
\ge -8(C+1)|z-\zeta|.
\end{align*}
Consequently,
\begin{align} \label{3.4}
|U^{\mu_w}(z) - U^{\mu_w}(\zeta)| \le 8(C+1)|z-\zeta|, \quad z\in
S_w,\ \zeta\in\C.
\end{align}
We now show that \eqref{3.4} is true for any $z,\zeta\in\C.$
Observe that
\[
\sup\{|U^{\mu_w}(z) - U^{\mu_w}(\zeta)|:|z-\zeta|\le\delta,\
z,\zeta\in\C\} = |U^{\mu_w}(z_0) - U^{\mu_w}(\zeta_0)|
\]
for some $z_0,\zeta_0\in\C,\ |z_0-\zeta_0|\le\delta,$ because
\[
\lim_{\stackrel{z,\zeta\to\infty}{|z-\zeta|\le\delta}}(U^{\mu_w}(z)
- U^{\mu_w}(\zeta))=0.
\]
Consider $h(\xi):=U^{\mu_w}(\xi) - U^{\mu_w}(\xi-z_0+\zeta_0),$
which is continuous on $\overline\C$ (cf. Theorem I.5.1 of
\cite{ST}) and harmonic in $\overline\C\setminus S_w.$ By the
maximum-minimum principle, we have
\begin{align*}
|U^{\mu_w}(z_0) - U^{\mu_w}(\zeta_0)|=|h(z_0)| &\le \max_{\xi\in
S_w} |h(\xi)| = |h(\xi_0)| \le 8(C+1)|z_0-\zeta_0|,
\end{align*}
where $\xi_0\in S_w,$ and where the last inequality follows from
\eqref{3.4}. Thus
\[
\sup\{|U^{\mu_w}(z) - U^{\mu_w}(\zeta)|:|z-\zeta|\le\delta,\
z,\zeta\in\C\} \le 8(C+1)\delta,
\]
i.e., $U^{\mu_w}$ is Lipschitz continuous in $\C.$
\end{proof}

\begin{proof}[Proof of Theorem \ref{thm1.1}]
The equilibrium equation \eqref{1.5} and Lemma \ref{lem3.1} give
that
\[
\int_0^{2\pi} f(t)\log\left|2\sin\frac{\theta-t}{2}\right|\,dt =
Q(\theta) - F_w, \qquad e^{i\theta}\in S_w.
\]
Differentiating this equation with respect to $\theta$, we obtain
that
\[
\frac{1}{2}\int_0^{2\pi} f(t)\cot\frac{\theta-t}{2}\,dt =
Q'(\theta)
\]
for almost every $\theta,\ e^{i\theta}\in S_w$.  The justification
of differentiation under the integral is done as in Lemma 2.45 of
\cite{DKM}. Indeed, consider the harmonic conjugate of $f$ (cf.
\cite[Chap. 3]{Gar})
\[
\tilde f(\theta) = \frac{1}{2\pi}\int_0^{2\pi}
f(t)\cot\frac{\theta-t}{2}\,dt,
\]
where the integral is understood in the principal value sense.
Since $f\in L_{\infty}\left([0,2\pi)\right),$ we have that $\tilde
f\in L_p\left([0,2\pi)\right)$ for any $p<\infty$, by M. Riesz's
theorem. From the Fundamental Theorem of Calculus and Fubini's
theorem, we obtain that
\[
\int_0^{2\pi} f(t)\log\left|2\sin\frac{\theta-t}{2}\right|\,dt =
\pi \int_0^{\theta} \tilde f(s)\,ds + c,\qquad e^{i\theta}\in \T,
\]
where $c$ is a constant. It follows that
\[
\pi \int_0^{\theta} \tilde f(s)\,ds = Q(\theta) - F_w - c,\qquad
e^{i\theta}\in S_w.
\]
Using the Fundamental Theorem of Calculus again, we can
differentiate the above equation, so that
\begin{align} \label{3.5}
\tilde f(\theta) = \frac{Q'(\theta)}{\pi} \qquad \mbox{a.e. on }
S_w.
\end{align}

Consider the analytic in $D$ function
\[
H(z):=\frac{1}{2\pi}\int_0^{2\pi}
f(t)\frac{e^{it}+z}{e^{it}-z}\,dt, \qquad z\in D,
\]
and recall that $H(z)=u(z)+i\tilde u(z), \ \tilde u(0)=0,$ where
$u$ and $\tilde u$ have the boundary values (cf. \cite{Gar})
\[
u|_{\T}=f \quad \mbox{ and } \quad \tilde u|_{\T}=\tilde f.
\]
Clearly, the function
\[
-iH^2(z)=2u(z)\tilde u(z) + i(\tilde u^2(z) - u^2(z)), \qquad z\in
D,
\]
is analytic in $D$. Hence the harmonic conjugate of $2u\tilde u$
is
\[
(2u\tilde u)\,\tilde{ } = \tilde u^2 - u^2 + c,
\]
where $c$ is selected by the standard convention $(2u\tilde
u)\,\tilde{ }(0) = 0 = \tilde u^2(0) - u^2(0) + c.$ Since $\tilde
u(0) = 0,$ we have
\[
c=u^2(0)=\left(\frac{1}{2\pi}\int_0^{2\pi} f(t)\,dt\right)^2 =
\left(\frac{\mu_w(\T)}{2\pi}\right)^2 = \frac{1}{4\pi^2}.
\]
Passing to the boundary values, we obtain that
\[
2(f\tilde f)\,\tilde{ } = \tilde f^2 - f^2 + \frac{1}{4\pi^2},
\qquad \mbox{a.e. on } \T,
\]
and that
\begin{align} \label{3.6}
\frac{2}{\pi} Q' \tilde f - \left(\frac{Q'}{\pi}\right)^2 -
2(f\tilde f)\,\tilde{ } + \frac{1}{4\pi^2} = f^2 - \left(\tilde f
- \frac{Q'}{\pi}\right)^2, \qquad \mbox{a.e. on } U.
\end{align}
Observe that the right hand side of \eqref{3.6} gives a
decomposition for the left hand side on $U$ into the positive part
$f^2$ and the negative part $-\left(\tilde f - Q'/\pi \right)^2,$
because of \eqref{3.5} and $f(t)=0,\ e^{it}\not\in S_w.$ Hence
\begin{align} \label{3.7}
f^2(\theta) &= \left(\frac{Q'(\theta)}{\pi}\right)^2 -
\frac{2}{\pi}(f Q')\tilde{ }\,(\theta) + \frac{1}{4\pi^2} \\
\nonumber &= \left(\frac{Q'(\theta)}{\pi}\right)^2 -
\frac{1}{\pi^2}\int_0^{2\pi} Q'(t) f(t)\cot\frac{\theta-t}{2}\,dt
+ \frac{1}{4\pi^2}
\end{align}
for a.e. $\theta$ such that $e^{i\theta}\in S_w.$

\end{proof}

\begin{proof}[Proof of Corollary \ref{cor1.2}]

Using \eqref{3.5}, we obtain that
\begin{align*}
&\frac{1}{\pi^2}\int_0^{2\pi} \left(Q'(\theta)-Q'(t)\right)
f(t)\cot\frac{\theta-t}{2}\,dt -
\left(\frac{Q'(\theta)}{\pi}\right)^2 + \frac{1}{4\pi^2} \\ &=
\frac{Q'(\theta)}{\pi^2} \int_0^{2\pi} f(t)\cot\frac{\theta-t}{2}
\,dt - \frac{1}{\pi^2}\int_0^{2\pi} Q'(t)
f(t)\cot\frac{\theta-t}{2}\,dt -
\left(\frac{Q'(\theta)}{\pi}\right)^2 + \frac{1}{4\pi^2} \\ &=
\left(\frac{Q'(\theta)}{\pi}\right)^2 -
\frac{1}{\pi^2}\int_0^{2\pi} Q'(t) f(t)\cot\frac{\theta-t}{2}\,dt
+ \frac{1}{4\pi^2} \qquad \mbox{for a.e. } e^{i\theta}\in S_w,
\end{align*}
which is the right hand side of \eqref{1.7}.

If $Q\in C^2(U)$ then $p\in C(U)$, because the function
\[
\Phi(\theta,t):=\left(Q'(\theta)-Q'(t)\right)
\cot\frac{\theta-t}{2}
\]
can be extended continuously to $U\times U$ by setting
$\Phi(\theta,\theta) := 2\, Q''(\theta).$ Hence $f$ has a
continuous extension on $S_w$ by \eqref{1.8}, which satisfies
$f(\theta)=\sqrt{p(\theta)},\ e^{i\theta}\in S_w.$ Using this
extension in \eqref{3.6}, we obtain that
\[
p=\frac{2}{\pi} Q' \tilde f - \left(\frac{Q'}{\pi}\right)^2 -
2(f\tilde f)\,\tilde{ } + \frac{1}{4\pi^2} = f^2 - \left(\tilde f
- \frac{Q'}{\pi}\right)^2
\]
everywhere on $U$. Therefore, $S_w=\overline{\{e^{i\theta}:
p(\theta)>0\}}$.

It is immediate to see now that the real analyticity of $Q(t)$ on
$U$ implies that of $p(t)$. If we assume that $S_w$ has infinitely
many arcs, then $p(t)$ must have infinitely many zeros on
$S_w\subset U$, and hence it must vanish identically.

\end{proof}

\begin{proof}[Proof of Corollary \ref{cor1.2a}]

Recall that by the Hilbert inversion formula we have
\[
(\tilde{f})\,\tilde{ } = - f + \frac{1}{2\pi} \int_0^{2\pi}
f(t)\,dt \qquad \mbox{a.e. on } \T,
\]
see \cite[Chap. III]{Gar} and \cite[$\S$28]{Mus}. But the latter
integral is equal to the mass $\mu_w(\T)=1.$ Hence we obtain from
\eqref{3.5} for $S_w=\T$ that
\begin{align*}
f(\theta) = \frac{1}{2\pi} - \frac{1}{\pi}
(Q')\,\tilde{}\,(\theta) = \frac{1}{2\pi} -
\frac{1}{2\pi^2}\int_0^{2\pi} Q'(t) \cot\frac{\theta-t}{2}\,dt,
\end{align*}
for a.e. $e^{it}\in\T.$ Since $Q'$ is H\"{o}lder continuous on
$\T$, we conclude that the same is true for its conjugate
$(Q')\,\tilde{}$ by \cite[Chap. III]{Gar}. Therefore, $f$ has a
 H\"{o}lder continuous extension to $\T$, and the above equation
 holds for all $\theta.$
\end{proof}

\begin{proof}[Proof of Proposition \ref{prop1.3}]

We first note that
\[
u(\theta) := U^{\mu_w}(e^{i\theta}) =
-\int\log|e^{i\theta}-e^{it}|\,d\mu_w(t) =
-\int\log\left|2\sin\frac{\theta-t}{2}\right|\,d\mu_w(t)
\]
is a strictly convex function of $\theta$ on each arc of
$\T\setminus S_w.$ Indeed, we have
\[
u''(\theta)=\frac{1}{4}\int\csc^2\frac{\theta-t}{2}\,d\mu_w(t) >0,
\quad e^{i\theta}\in\T\setminus S_w,
\]
so that the claim follows. Suppose that
$S_w\cap(e^{i\alpha},e^{i\beta})$ is not an arc. Then there are
two points $\theta_1,\theta_2 \in (\alpha,\beta)$ such that
$e^{i\theta_1},e^{i\theta_2}\in S_w$, and $e^{i\theta}\not\in S_w$
for $\theta_1<\theta<\theta_2$. Since $u(\theta)+Q(\theta)$ is a
strictly convex function on $(\theta_1,\theta_2)$, which takes
values $u(\theta_1)+Q(\theta_1)=u(\theta_2)+Q(\theta_2)=F_w$ by
\eqref{1.5}, we have that $u(\theta)+Q(\theta) < F_w$ for
$\theta\in(\theta_1,\theta_2)$. But this contradicts \eqref{1.4},
which is true for any $e^{i\theta}\in\T.$

\end{proof}

\begin{proof}[Proof of Theorem \ref{thm1.4}]

We start with the observation that
\[
\frac{e^{it}+e^{i\theta}}{e^{it}-e^{i\theta}}=i\cot\frac{\theta-t}{2}.
\]
Hence we obtain for the singular Schwarz integral (understood as
the principal value) that
\[
\frac{1}{2\pi}\int_0^{2\pi} f(t) \frac{e^{it}+e^{i\theta}}
{e^{it}-e^{i\theta}} dt = i\tilde f(\theta),\qquad \mbox{for a.e.
} \theta\in[0,2\pi).
\]
It follows from \eqref{3.5} that
\[
\frac{1}{2\pi}\int_0^{2\pi} f(t) \frac{e^{it}+e^{i\theta}}
{e^{it}-e^{i\theta}} dt = \frac{i}{\pi} Q'(\theta), \qquad
\mbox{a.e. on } S_w,
\]
and
\[
\frac{1}{\pi i}\int_0^{2\pi} \frac{f(t)\,d(e^{it})}
{e^{it}-e^{i\theta}} - \frac{1}{2\pi}= \frac{i}{\pi} Q'(\theta),
\qquad \mbox{a.e. on } S_w.
\]
Consequently, the density function $f$ satisfies the following
singular integral equation with Cauchy kernel:
\begin{align} \label{3.8}
\frac{1}{\pi i}\int_{S_w} \frac{f(z)\,dz}{z-\zeta} = \frac{i}{\pi}
Q'(\theta) + \frac{1}{2\pi}, \qquad \zeta=e^{i\theta} \in S_w,
\end{align}
where we set $f(z)=f(e^{it}):=f(t).$ Since $S_w$ consists of
finitely many arcs $\Gamma_k,\ k=1,\ldots,K,$ and $f$ is a
continuous function vanishing at the endpoints of $S_w$ by
Corollary \ref{cor1.2}, we obtain from the results of \cite[Chap.
11, $\S$88]{Mus} (see also \cite{Mu41}) that $f$ must be the
unique solution of \eqref{3.8} given by the singular integral
\begin{align} \label{3.9}
f(z)=\frac{\sqrt{R(z)}}{\pi i} \int_{S_w}
\frac{g(\zeta)\,d\zeta}{\sqrt{R(\zeta)}(\zeta-z)}, \qquad z \in
S_w,
\end{align}
where $g(\zeta)$ denotes the right hand side of \eqref{3.8}, and
$\sqrt{R(z)}$ is defined in the statement of Theorem \ref{thm1.4}.
Furthermore, vanishing of $f$ at the endpoints of $S_w$ implies
that $g$ must satisfy the following moment conditions
\begin{align*}
\int_{S_w} \frac{z^k g(z)\,dz}{\sqrt{R(z)}} = 0, \qquad
k=0,\ldots,K-1,
\end{align*}
see \cite[pp. 251, 256]{Mus} and \cite[p. 14]{Mu41}. Hence
\eqref{1.9}-\eqref{1.11} are proved. Equation \eqref{1.12} simply
expresses the fact that $d\mu_w(e^{it})=f(t)\,dt$ is a probability
measure.

Observe that the equilibrium equation \eqref{1.5} gives
\begin{align} \label{3.10}
U^{\mu_{w}}(a_{k+1})-U^{\mu_{w}}(b_k) =
Q(\beta_k)-Q(\alpha_{k+1}), \qquad k=1,\ldots,K.
\end{align}
On the other hand, we have that
\begin{align*}
U^{\mu_{w}}(a_{k+1})-U^{\mu_{w}}(b_k) &=
\int_{\beta_k}^{\alpha_{k+1}} \frac{d}{dt}\, U^{\mu_{w}}(e^{it})\
dt \\ &= -\frac{1}{2} \int_{\beta_k}^{\alpha_{k+1}} \int_0^{2\pi}
f(u)\cot\frac{t-u}{2}\,du\ dt \\ &= -\frac{1}{2i}
\int_{\beta_k}^{\alpha_{k+1}} \int_0^{2\pi} f(u)
\frac{e^{iu}+e^{it}}{e^{iu}-e^{it}}\,du\ dt \\ &= -\frac{\pi}{i}
\int_{\beta_k}^{\alpha_{k+1}} \left( \frac{1}{\pi i} \int_{S_w}
\frac{F(e^{iu})\,d(e^{iu})}{e^{iu}-e^{it}} - \frac{1}{2\pi}
\right) dt.
\end{align*}
Note that $\dis\lim_{r\to 1-} \sqrt{R(re^{iu})} = - \lim_{r\to 1+}
\sqrt{R(re^{iu})}$ for $e^{iu}\in S_w.$ Thus the limiting boundary
values of $\sqrt{R(z)}$ on the arcs of $S_w$, from inside and
outside the unit circle, are opposite in sign. Hence the same is
true for the function $F(z)$, which is analytic in
$\overline\C\setminus S_w$. Passing to the contour integral over
both sides of the cut $S_w$, we obtain by the Cauchy integral
theorem that
\begin{align*}
\frac{1}{\pi i} \int_{S_w}
\frac{F(e^{iu})\,d(e^{iu})}{e^{iu}-e^{it}} = \frac{1}{2\pi i}
\oint_{S_w} \frac{F(e^{iu})\,d(e^{iu})}{e^{iu}-e^{it}} = F(e^{it})
- \lim_{z\to\infty} F(z).
\end{align*}
The latter limit is equal to 0, in fact, which follows from the
moment conditions \eqref{1.11}. Indeed, we have in a neighborhood
of $\infty$ that
\begin{align*}
F(z) &= \frac{\sqrt{R(z)}}{\pi i} \int_{S_w}
\frac{g(\zeta)\,d\zeta}{\sqrt{R(\zeta)}(\zeta-z)} =
-\frac{\sqrt{R(z)}}{\pi i} \sum_{k=0}^{\infty} z^{-k-1} \int_{S_w}
\frac{\zeta^k g(\zeta)\,d\zeta}{\sqrt{R(\zeta)}} \\ &=
-\frac{\sqrt{R(z)}}{\pi i} \sum_{k=K}^{\infty} z^{-k-1} \int_{S_w}
\frac{\zeta^k g(\zeta)\,d\zeta}{\sqrt{R(\zeta)}},
\end{align*}
which gives $F(\infty)=0.$ Consequently,
\begin{align*}
U^{\mu_{w}}(a_{k+1})-U^{\mu_{w}}(b_k) &= -\frac{\pi}{i}
\int_{\beta_k}^{\alpha_{k+1}} \left( F(e^{it}) - \frac{1}{2\pi}
\right) dt \\ &= -\frac{\pi}{i} \int_{\beta_k}^{\alpha_{k+1}}
F(e^{it})\,dt + \frac{\alpha_{k+1}-\beta_k}{2i}.
\end{align*}
Combining this equation with \eqref{3.10}, we prove \eqref{1.13}.

\end{proof}

\begin{proof}[Proof of Theorem \ref{thm2.1}]

Note that $S_w$ cannot contain the zeros of $w$ by \eqref{1.5},
i.e., we have that
\[
Q(\theta) = -\sum_{j=1}^J \lambda_j \log|e^{i\theta}-z_j|
\]
is infinitely differentiable in a neighborhood of $S_w$ in $\T.$
Thus the results of Section 1 apply here. It is clear that
\[
Q'(\theta) = - \sum_{j=1}^J \frac{\lambda_j r_j
\sin(\theta-\phi_j)}{1+r_j^2-2 r_j \cos(\theta-\phi_j)} = -
\sum_{j=1}^J \frac{\lambda_j r_j
\sin(\theta-\phi_j)}{|e^{i\theta}-z_j|^2},
\]
where $z_j=r_j e^{i\phi_j},\ j=1,\ldots,J.$ Applying elementary
trigonometric identities, we obtain that
\begin{align*}
Q'(\theta) - Q'(t) = \sum_{j=1}^J \frac{2\lambda_j r_j^2
\sin(\theta-t) - 2\lambda_j r_j (1+r_j^2)
\cos\left(\frac{\theta+t}{2}-\phi_j\right) \sin\frac{\theta-t}{2}
}{|e^{i\theta}-z_j|^2 |e^{it}-z_j|^2}
\end{align*}
and that
\begin{align*}
&\left(Q'(\theta) - Q'(t)\right) \cot\frac{\theta-t}{2} \\ &=
\sum_{j=1}^J \frac{2\lambda_j r_j^2 (\cos(\theta-t)+1) - \lambda_j
r_j (1+r_j^2) (\cos(\theta-\phi_j) +
\cos(t-\phi_j))}{|e^{i\theta}-z_j|^2 |e^{it}-z_j|^2}.
\end{align*}
If we insert the above representations in \eqref{1.8}, it becomes
clear that
\[
f^2(\theta) = \frac{L_{2J}(e^{i\theta})}{\pi^2 \prod_{j=1}^J
|e^{i\theta}-z_j|^4}, \qquad e^{i\theta} \in S_w,
\]
where $L_{2J}(e^{i\theta})$ is a trigonometric polynomial of
degree at most $2J.$ Since $L_{2J}(e^{i\theta})$ has at most $4J$
zeros on $[0,2\pi),$ and $S_w=\overline{\{e^{i\theta}:
L_{2J}(e^{i\theta})>0\}}$, we conclude that $S_w$ consists of at
most $2J$ arcs of $\T$. Naturally, $L_{2J}$ and $f$ vanish at the
endpoints of those arcs.

Alternatively, we can write
\begin{align*}
Q'(\theta) &= - \sum_{j=1}^J \frac{\lambda_j r_j
\sin(\theta-\phi_j)}{1+r_j^2-2 r_j \cos(\theta-\phi_j)}=
\frac{1}{2i} \sum_{j=1}^J \frac{\lambda_j (z_j e^{-i\theta} - \bar
z_j e^{i\theta})}{|e^{i\theta}-z_j|^2} \\ &= \frac{1}{2i}
\sum_{j=1}^J \lambda_j \frac{z_j - \bar z_j \zeta^2}
{(\zeta-z_j)(1-\bar z_j \zeta)},\qquad \zeta=e^{i\theta}.
\end{align*}
Thus we obtain from Theorem \ref{thm1.4} that
\begin{align*}
g(\zeta) &= \frac{1}{2\pi} \sum_{j=1}^J \lambda_j \frac{z_j-\bar
z_j \zeta^2}{(\zeta-z_j)(1-\bar z_j \zeta)} + \frac{1}{2\pi} \\
\nonumber &= \frac{1+\sum_{j=1}^J \lambda_j}{2\pi} +
\frac{1}{2\pi} \sum_{j=1}^J \frac{\lambda_jz_j}{\zeta-z_j} +
\frac{1}{2\pi} \sum_{j=1}^J \frac{\lambda_j \bar
z_j^{-1}}{\zeta-\bar z_j^{-1}}.
\end{align*}
This form of $g$ is convenient for evaluation of the integrals as
in \eqref{1.9}.

We first consider the case $S_w=\T$. Then we have from \eqref{3.8}
that
\[
\frac{1}{\pi i}\int_{\T} \frac{f(z)\,dz}{z-\zeta} = g(\zeta),
\qquad \zeta \in \T.
\]
Applying the inversion formula to this Cauchy singular integral,
we obtain
\[
f(z)=\frac{1}{\pi i}\int_{\T} \frac{g(\zeta)\,d\zeta}{\zeta-z},
\qquad z \in \T,
\]
by $\S27$ of \cite{Mus}. It follows from Plemelj's formulas (cf.
\cite[$\S$17]{Mus}) that
\[
f(z) = \lim_{\stackrel{|\xi|<1}{\xi\to z}} \frac{1}{2\pi
i}\int_{\T} \frac{g(\zeta)\,d\zeta}{\zeta-\xi} +
\lim_{\stackrel{|\xi|>1}{\xi\to z}} \frac{1}{2\pi i}\int_{\T}
\frac{g(\zeta)\,d\zeta}{\zeta-\xi},
\]
where both integrals are taken in the counterclockwise direction.
Evaluating the above integrals via residues and passing to the
limits, we immediately have that
\begin{align*}
f(z) = \frac{1+\sum_{j=1}^J \lambda_j}{2\pi} &+ \frac{1}{2\pi}
\sum_{|z_j|>1} \frac{\lambda_jz_j}{z-z_j} + \frac{1}{2\pi}
\sum_{|z_j|<1} \frac{\lambda_j \bar z_j^{-1}}{z-\bar z_j^{-1}} \\
&- \frac{1}{2\pi} \sum_{|z_j|<1} \frac{\lambda_jz_j}{z-z_j} -
\frac{1}{2\pi} \sum_{|z_j|>1} \frac{\lambda_j \bar
z_j^{-1}}{z-\bar z_j^{-1}} \\ = \frac{1+\sum_{j=1}^J
\lambda_j}{2\pi} &+ \frac{1}{2\pi} \sum_{|z_j|>1} \lambda_j \left(
\frac{z_j}{z-z_j}-\frac{\bar z_j^{-1}}{z-\bar z_j^{-1}}\right) \\
&+ \frac{1}{2\pi} \sum_{|z_j|<1} \lambda_j \left(\frac{\bar
z_j^{-1}}{z-\bar z_j^{-1}}-\frac{z_j}{z-z_j}\right).
\end{align*}
Thus \eqref{2.2} follows by a simple algebraic manipulation.
Another proof of \eqref{2.2} can be produced by using Corollary
\ref{cor1.2a} and finding the harmonic conjugate of $Q'$ from its
trigonometric form.

We now assume that $S_w$ consists of $K\ge 1$ proper arcs of $\T$,
and apply \eqref{1.9} of Theorem \ref{thm1.4}. As in the proof of
Theorem \ref{thm1.4}, we observe that the limiting boundary values
of $\sqrt{R(\zeta)}$ on those arcs, from inside and outside the
unit circle, are opposite in sign. Hence the same is true for the
function $g(\zeta)/(\sqrt{R(\zeta)}(\zeta-z))$, which is analytic
in $\overline\C\setminus S_w$ except for the simple poles at $z_j$
and $\bar z_j^{-1}.$ Here, we used the natural extension of
$g(\zeta)$ from $S_w$ to $\C$ as a rational function. Passing to
the contour integral over both sides of the cut $S_w$, and
computing the residues at $z_j$ and $\bar z_j^{-1},$ we obtain
that
\begin{align*}
f(z) &= \frac{\sqrt{R(z)}}{2\pi i} \oint_{S_w}
\frac{g(\zeta)\,d\zeta}{\sqrt{R(\zeta)}(\zeta-z)} \\ &=
\frac{\sqrt{R(z)}}{2\pi} \sum_{j=1}^J
\frac{\lambda_jz_j}{\sqrt{R(z_j)}(z_j-z)} +
\frac{\sqrt{R(z)}}{2\pi} \sum_{j=1}^J \frac{\lambda_j \bar
z_j^{-1}}{\sqrt{R(\bar z_j^{-1})}(\bar z_j^{-1}-z)}, \quad z \in
S_w.
\end{align*}
Thus \eqref{2.3} is proved. Our next goal is to show that the
number of intervals for $S_w$ is at most $J$. Observe that the
previous equation gives
\begin{align*}
f(t) &= \frac{\sqrt{R(e^{it})}}{\pi} \sum_{j=1}^J
\frac{\lambda_j}{2} \left( \frac{-z_j}{\sqrt{R(z_j)}(e^{it}-z_j)}
+ \frac{1}{\sqrt{R(\bar z_j^{-1})}(1-\bar z_j e^{it})} \right) \\
&= \frac{\sqrt{R(e^{it})}P(e^{it})}{\pi\prod_{j=1}^J
(e^{it}-z_j)(1-\bar z_j e^{it})} =
\frac{\sqrt{R(e^{it})}P(e^{it})}{\pi e^{iJt} \prod_{j=1}^J
|e^{it}-z_j|^2}, \quad e^{it} \in S_w,
\end{align*}
where $P$ is an algebraic polynomial. Comparing this with the
previously obtained form for $f^2(t)$, we conclude that
\begin{align} \label{3.11}
R(e^{it})P^2(e^{it})=e^{i2Jt}L_{2J}(e^{it}),
\end{align}
for $e^{it} \in S_w,$ so that these polynomials in $e^{it}$
coincide. It follows that $\deg(RP^2)\le 4J,$ i.e., $2K+2\deg P
\le 4J$ and
\begin{align} \label{3.12}
K + \deg P \le 2J.
\end{align}
Since $S_w=\overline{\{e^{i\theta}: L_{2J}(e^{it})>0\}}$, we have
that $L_{2J}(e^{it})$ takes real values for all $t$, and that
$L_{2J}(e^{it})\le 0$ for $e^{it}\in \T\setminus S_w.$ Equation
\eqref{3.11} and the one before it suggest that
\begin{align} \label{3.13}
\sqrt{R(e^{it})}P(e^{it})e^{-iJt} \ge 0, \qquad e^{it} \in S_w,
\end{align}
and that $\sqrt{R(e^{it})}P(e^{it})e^{-iJt}$ is pure imaginary on
$\T\setminus S_w.$ Assume that $P(e^{it})\neq 0$ for $t\in
(\beta_k, \alpha_{k+1}),$ that is $P$ does not vanish on $\T$
between a pair of neighboring arcs of $S_w.$ Note that the
argument of $\sqrt{R(e^{it})}$ decreases by $\pi/2$ when we pass
over an endpoint of $S_w$ (when moving in the positive direction
on $\T$), while the argument of $P(e^{it})e^{-iJt}$ remains
continuous everywhere on $\T$, except for the zeros of $P$. Hence
we have that the values of $\sqrt{R(e^{it})}P(e^{it})e^{-iJt}$
should have opposite signs on these neighboring arcs of $S_w,$
which immediately contradicts \eqref{3.13}. It follows that
$P(e^{it})$ has a zero on $\T$ between each pair of arcs of $S_w,$
so that $\deg P \ge K.$ Finally, $2K \le K + \deg P \le 2J,$ by
\eqref{3.12}, and $K \le J.$

In the remaining part, we prove \eqref{2.4} and \eqref{2.5}.
Applying Theorem \ref{thm1.4}, we see that $g$ must satisfy the
moment conditions of \eqref{1.11}. Again, these integrals are
found by passing to the contour integrals over both sides of the
cut $S_w$, and computing the residues at $z_j$, $\bar z_j^{-1}$
and $\infty$:
\begin{align*}
\frac{1}{\pi i} \int_{S_w} \frac{z^k g(z)\,dz}{\sqrt{R(z)}} &=
\frac{1}{2\pi i} \oint_{S_w} \frac{z^k g(z)\,dz}{\sqrt{R(z)}} \\
&= \frac{1}{2\pi} \sum_{j=1}^J \lambda_j \left(
\frac{z_j^{k+1}}{\sqrt{R(z_j)}} + \frac{\bar z_j^{-k-1}}
{\sqrt{R(\bar z_j^{-1})}} \right) \\ & - \frac{1+\sum_{j=1}^J
\lambda_j}{2\pi} \lim_{z\to\infty} \frac{z^{k+1}}{\sqrt{R(z)}}.
\end{align*}
Note that
\[
\lim_{z\to\infty} \frac{z^{k+1}}{\sqrt{R(z)}} = \left\{
\begin{array}{ll}
0,\quad k=0,\ldots,K-2,\\
1,\quad k=K-1,
\end{array}
\right.
\]
so that \eqref{2.4} follows from \eqref{1.11}.

We deduce \eqref{2.5} from \eqref{1.13}. For this purpose, we
evaluate $F(z)$ for $z\in\C\setminus S_w,$ by using the residues
at $z_j$, $\bar z_j^{-1}$ and $z$:
\begin{align*}
F(z) &= \frac{\sqrt{R(z)}}{2\pi i} \oint_{S_w}
\frac{g(\zeta)\,d\zeta}{\sqrt{R(\zeta)}(\zeta-z)} \\ &=
\frac{\sqrt{R(z)}}{2\pi} \sum_{j=1}^J
\frac{\lambda_jz_j}{\sqrt{R(z_j)}(z_j-z)} +
\frac{\sqrt{R(z)}}{2\pi} \sum_{j=1}^J \frac{\lambda_j \bar
z_j^{-1}}{\sqrt{R(\bar z_j^{-1})}(\bar z_j^{-1}-z)} + g(z) \\ &=
F_1(z) + g(z).
\end{align*}
Hence
\begin{align} \label{3.14}
\int_{\beta_k}^{\alpha_{k+1}} F_1(e^{it})\, dt +
\int_{\beta_k}^{\alpha_{k+1}} g(e^{it})\, dt =
\frac{\alpha_{k+1}-\beta_k}{2\pi} +
\frac{Q(\alpha_{k+1})-Q(\beta_k)}{\pi}\,i,
\end{align}
where $k=1,\ldots,K,$ by \eqref{1.13}. We next compute that
\begin{align*}
\int_{\beta_k}^{\alpha_{k+1}} g(e^{it})\, dt &=
\int_{\beta_k}^{\alpha_{k+1}} \left(\frac{1}{2\pi} \sum_{j=1}^J
\lambda_j\left(\frac{z_j}{e^{it}-z_j} + \frac{\bar z_j^{-1}}
{e^{it}-\bar z_j^{-1}}\right) + \frac{1+\sum_{j=1}^J
\lambda_j}{2\pi} \right) dt \\ &= \frac{1}{2\pi i} \sum_{j=1}^J
\lambda_j\left(z_j\int_{b_k}^{a_{k+1}} \frac{dz}{z(z-z_j)} + \bar
z_j^{-1}\int_{b_k}^{a_{k+1}} \frac{dz} {z(z-\bar z_j^{-1})}\right)
\\  &+ \frac{\alpha_{k+1}-\beta_k}{2\pi} \left(1+\sum_{j=1}^J
\lambda_j\right).
\end{align*}
It is immediate to see that
\begin{align*}
&\frac{1}{2\pi i} \sum_{j=1}^J
\lambda_j\left(z_j\int_{b_k}^{a_{k+1}} \frac{dz}{z(z-z_j)} + \bar
z_j^{-1}\int_{b_k}^{a_{k+1}} \frac{dz} {z(z-\bar z_j^{-1})}\right)
\\ &= \sum_{j=1}^J \frac{\lambda_j}{2\pi i}
\left(\log\frac{a_{k+1}-z_j}{b_k-z_j} + \log\frac{a_{k+1}-\bar
z_j^{-1}}{b_k-\bar z_j^{-1}} - 2 \log\frac{a_{k+1}}{b_k}\right) \\
&= \sum_{j=1}^J \frac{\lambda_j}{2\pi i}
\left(\log\frac{a_{k+1}|a_{k+1}-z_j|^2}{b_k|b_k-z_j|^2} - 2
\log\frac{a_{k+1}}{b_k}\right) \\ &=
\frac{Q(\alpha_{k+1})-Q(\beta_k)}{\pi}\,i -
\frac{\alpha_{k+1}-\beta_k}{2\pi} \sum_{j=1}^J \lambda_j.
\end{align*}
Substituting the results of the above computations in
\eqref{3.14}, we immediately obtain \eqref{2.5}.

\end{proof}

\begin{proof}[Proof of Theorem \ref{thm2.2}]

We follow essentially the same scheme as in the previous proof.
Note that $Q(\theta) = t_M(\theta)$ is infinitely differentiable
on $\T.$ It is clear that
\[
Q'(\theta) = \sum_{m=-M}^M i m c_m e^{im\theta}
\]
and
\[
g(e^{i\theta})=\frac{1}{2\pi} - \frac{1}{\pi} \sum_{m=-M}^M m c_m
e^{im\theta},
\]
where $g$ is defined in Theorem \ref{thm1.4}. Inserting
$Q'(\theta)$ into \eqref{1.8}, we observe that the right hand side
$p(\theta)$ is a trigonometric polynomial of degree at most $2M$.
Thus we have from Theorem \ref{thm1.1} and Corollary \ref{cor1.2}
that $d\mu_w(e^{i\theta})=f(\theta)\,d\theta$ with
$f^2(\theta)=L_{2M}(e^{i\theta})$, where $L_{2M}(z)$ is a Laurent
polynomial of degree at most $2M.$ It follows that $S_w$ consists
of at most $2M$ arcs, because $L_{2M}$ has at most $4M$ zeros on
$\T$.

When $S_w=\T$, we obtain \eqref{2.7} from Corollary \ref{cor1.2a},
after substituting the conjugate
\[
(Q')\,\tilde{}(\theta) = \frac{1}{2\pi}\int_0^{2\pi} Q'(t)
\cot\frac{\theta-t}{2}\,dt = \sum_{m=-M}^M m\, {\rm sgn}(m)\, c_m
e^{im\theta}.
\]
Suppose that $S_w$ consists of $K\ge 1$ proper arcs of $\T$, and
apply \eqref{1.9} of Theorem \ref{thm1.4}. Using that the limiting
boundary values of $\sqrt{R(\zeta)}$ on $S_w$ are opposite in
sign, we again pass to the contour integral over both sides of the
cut $S_w$:
\begin{align} \label{3.15}
F_2(z) = F(z) &= \frac{\sqrt{R(z)}}{2\pi i} \oint_{S_w}
\frac{g(\zeta)\,d\zeta}{\sqrt{R(\zeta)}(\zeta-z)} \\ \nonumber &=
\frac{\sqrt{R(z)}}{2\pi i} \oint_{S_w}
\frac{1}{\sqrt{R(\zeta)}(\zeta-z)} \left( \frac{1}{2\pi} -
\frac{1}{\pi} \sum_{m=-M}^M m c_m \zeta^m \right)\,d\zeta \\
\nonumber &= - \frac{\sqrt{R(z)}}{\pi} \sum_{m=-M}^M \frac{m
c_m}{2\pi i} \oint_{S_w}
\frac{\zeta^m\,d\zeta}{\sqrt{R(\zeta)}(\zeta-z)}, \quad z \in S_w.
\end{align}
The latter contour integrals are found by the Cauchy integral
theorem and evaluation of residues at $\infty$ and at $0,$ with
the help of the series expansions for
$1/(\sqrt{R(\zeta)}(\zeta-z))$. This immediately gives the stated
form of $F_2(z)=F(z),\ z\in S_w,$ so that \eqref{2.8} follows from
\eqref{1.10}. The expansion coefficients $r_k(z)$ and $s_k(z)$ can
be expressed in the standard way
\[
r_k(z) = \frac{1}{k!}\, \frac{d^k}{d\zeta^k} \left. \left(
\frac{1}{\sqrt{R(\zeta)}(\zeta-z)} \right) \right|_{\zeta=0},\quad
k\ge 0,
\]
and
\[
s_k(z) = \frac{1}{k!}\, \frac{d^k}{d\zeta^k} \left. \left(
\frac{\zeta^{K+1}}{\sqrt{\zeta^{2K}R(1/\zeta)}(1-z\zeta)} \right)
\right|_{\zeta=0},\quad k\ge 0.
\]
It transpires now that $s_k(z)$ is a polynomial in $z$ of degree
at most $k$, and that $r_k(z)$ is a polynomial in $1/z$ of degree
at most $k+1$. Hence we have
\[
F_2(z) =\frac{P(z)\sqrt{R(z)}}{z^M}, \quad z\in S_w,
\]
where $P(z)$ is a polynomial in $z$. Comparing this with the
previously obtained form for $f^2(t)$, we conclude that
\begin{align*}
R(e^{it})P^2(e^{it})=e^{i2Mt}L_{2M}(e^{it}),\qquad e^{it} \in S_w,
\end{align*}
so that these polynomials in $e^{it}$ coincide. It follows that
$\deg(RP^2)\le 4M,$ i.e., $2K+2\deg P \le 4M$ and
\begin{align} \label{3.16}
K + \deg P \le 2M.
\end{align}
Since $S_w=\overline{\{e^{i\theta}: L_{2M}(e^{it})>0\}}$, we have
that $L_{2M}(e^{it})$ takes real values for all $t$, and that
$L_{2M}(e^{it})\le 0$ for $e^{it}\in \T\setminus S_w.$ Therefore,
\begin{align} \label{3.17}
\sqrt{R(e^{it})}P(e^{it})e^{-iMt} \ge 0, \qquad e^{it} \in S_w,
\end{align}
being the density of $\mu_w$, and
$\sqrt{R(e^{it})}P(e^{it})e^{-iMt}$ is pure imaginary on
$\T\setminus S_w.$ Assume that $P$ does not vanish on $\T$ between
a pair of neighboring arcs of $S_w.$ Note again that the argument
of $\sqrt{R(e^{it})}$ decreases by $\pi/2$ when we pass over an
endpoint of $S_w$, while the argument of $P(e^{it})e^{-iMt}$ is
continuous on $\T$, except for possible zeros of $P$. Hence the
values of $\sqrt{R(e^{it})}P(e^{it})e^{-iMt}$ should have opposite
signs on these neighboring arcs of $S_w,$ which immediately
contradicts \eqref{3.17}. It follows that $P(e^{it})$ has a zero
on $\T$ between each pair of arcs of $S_w,$ so that $\deg P \ge
K.$ Finally, $2K \le K + \deg P \le 2M,$ by \eqref{3.16}, and $K
\le M.$

Repeating the same evaluation as in \eqref{3.15}, but for
$z\not\in S_w$, we obtain that
\begin{align*}
F(z) &= \frac{\sqrt{R(z)}}{2\pi i} \oint_{S_w}
\frac{g(\zeta)\,d\zeta}{\sqrt{R(\zeta)}(\zeta-z)} = F_2(z) + g(z).
\end{align*}

\end{proof}

{\bf Acknowledgement.} The author thanks the referee for numerous
suggested improvements.

\end{document}